\documentclass{amsart}
\usepackage{amssymb}
\usepackage{amsfonts}
\usepackage{amssymb}
\usepackage{amsmath}
\usepackage{amsthm}
\usepackage{enumerate}
\usepackage{tabularx}
\usepackage{centernot}
\usepackage{mathtools}
\usepackage{stmaryrd}
\usepackage{amsthm,amssymb}
\usepackage{etoolbox}
\usepackage{tikz}
\usepackage{amssymb}
\usetikzlibrary{matrix}
\usepackage{tikz-cd}
\usepackage{tikz}
\usepackage{marginnote}
\definecolor{mygray}{gray}{0.85}
\usepackage[backgroundcolor=mygray,colorinlistoftodos,prependcaption,textsize=small]{todonotes}
\usepackage{xargs}                      

\renewcommand{\leq}{\leqslant}

\makeatletter
\def\subsection{\@startsection{subsection}{3}%
  \z@{.5\linespacing\@plus.7\linespacing}{.3\linespacing}%
  {\bfseries\centering}}
\makeatother

\makeatletter
\def\subsubsection{\@startsection{subsubsection}{3}%
  \z@{.5\linespacing\@plus.7\linespacing}{.3\linespacing}%
  {\centering}}
\makeatother

\makeatletter
\def\myfnt{\ifx\protect\@typeset@protect\expandafter\footnote\else\expandafter\@gobble\fi}
\makeatother

\newtheorem{theorem}{Theorem}

\newtheorem{definition}[theorem]{Definition}

\newtheorem{fact}[theorem]{Fact}

\newtheorem{notation}[theorem]{Notation}

\newtheorem{cclaim}[theorem]{Claim}

\newcounter{claimcounter}
\numberwithin{claimcounter}{theorem}

\begin{document}

\begin{abstract} For every countable structure $M$ we construct an $\aleph_0$-stable countable structure $N$ such that $Aut(M)$ and $Aut(N)$ are topologically isomorphic. This shows that it is impossible to detect any form of stability of a countable structure $M$ from the topological properties of the Polish group $Aut(M)$.
\end{abstract}

\title{Automorphism Groups of Countable Stable Structures}
\thanks{Partially supported by European Research Council grant 338821. No. 1107 on Shelah's publication list.}

\author{Gianluca Paolini}
\address{Einstein Institute of Mathematics,  The Hebrew University of Jerusalem, Israel}

\author{Saharon Shelah}
\address{Einstein Institute of Mathematics,  The Hebrew University of Jerusalem, Israel \and Department of Mathematics,  Rutgers University, U.S.A.}

\date{\today}
\maketitle

\section{Introduction}

	As well known, the non-Archimedean Polish groups--those Polish groups admitting a basis at the identity of open subgroups--are precisely the Polish groups that can be represented as the automorphism groups of countable structures. A common theme of the last decades has been the search for connections between model theoretic properties of such structures and properties of their automorphism groups. 
	
	For example, the result of Engeler, Ryll-Nardzewski and Svenonius stating that the theory of a countable structure is countably categorical if and only if its automorphism group is oligomorphic \cite{1, 2, 3}, or the theorem of Ahlbrandt and Ziegler stating that two countable structures are bi-interpretable if and only if their automorphism groups are topologically isomorphic \cite{ziegler}. For more advanced result in this direction dealing with reconstruction up to bi-definability see \cite{Sh_Pa_recon, rubin}.
	

	Perhaps the pre-eminent model theoretic property is {\em stability}. In the present study we show that any attempt at a topological characterization of the group of automorphisms of a countable {\em stable} structure is doomed to fail. More strongly:
	
	\begin{theorem}\label{main_th} For every countable\footnote{In the present paper we consider only structures in a countable language.} structure $M$ there exists an $\aleph_0$-stable countable structure $N_M = N$ such that $Aut(M)$ and $Aut(N)$ are topologically isomorphic with respect to the naturally associated Polish group topologies.
\end{theorem}
	
	In order to witness the continuity of the isomorphism constructed in the proof of Theorem \ref{main_th} we use a new notion of interpretability, which we call $\mathfrak{L}_{\omega_1, \omega}$-semi-interpretability. In fact, in our proof, given a countable structure $M$, we construct an $\aleph_0$-stable structure $N_M = N$ and show that not only there is an isomorphism of  topological groups $\alpha: Aut(M) \rightarrow Aut(N)$, but that this $\alpha$ can be chosen to be such that it is induced by the map witnessing that $N$ is $\mathfrak{L}_{\omega_1, \omega}$-semi-interpretabile in $M$. Although the continuity of the isomorphism constructed in the proof of Theorem~\ref{main_th} is evident from the construction we believe that this new notion of interpretability is interesting per se, and that it gives more canonicity to our construction.
	
	Finally, the theory $Th(N_M)$ of Theorem \ref{main_th} can be shown to be NDOP and NOTOP, but this will not be proved here, since it is outside of the scope of this study.
	

\section{Proofs}
	
	 To make the exposition complete we first introduce the classical notion of first-order interpretability (Definition \ref{classical_interpretability}), and then define the notion of $\mathfrak{L}_{\omega_1, \omega}$-semi-interpretability (Definition \ref{semi_interpretability}). After this, we state two facts (Facts~\ref{semi_int_fact} and \ref{cont_homo}) which will be crucially used in the proof of Theorem~\ref{main_th} and then proceed to the proof.
	
	\begin{definition}\label{classical_interpretability} Let $M$ and $N$ be models. We say that $N$ is interpretable in $M$ if for some $n < \omega$ there are:
	\begin{enumerate}[(1)]
	\item a $\emptyset$-definable subset $D$ of $M^n$;
	\item a $\emptyset$-definable equivalence relation $E$ on $D$;
	\item a bijection $\alpha: N \rightarrow D/E$ such that for every $m < \omega$ and $\emptyset$-definable subset $R$ of $N^m$ the subset of $M^{nm}$ given by:
	$$\hat{R} = \{ (\bar{a}_1, ..., \bar{a}_m) \in (M^n)^m : (\alpha^{-1}(\bar{a}_1/E), ...,  \alpha^{-1}(\bar{a}_m/E)) \in R\}$$
is $\emptyset$-definable in $M$.
\end{enumerate}
\end{definition}

	\begin{notation} Let $\tau$ be a language.
	\begin{enumerate}[(1)]
	\item For $R \in \tau$ a predicate, we denote by $k(R)$ the arity of $R$.
	\item Given a $\tau$-structure $M$ and a $\tau$-formula $\varphi(\bar{x}) = \varphi(x_0, ..., x_{n-1})$, we let:
	$$\varphi(M) = \{ \bar{a} \in M^n : M \models \varphi(\bar{a})\}.$$
	\item Given a $\tau$-structure $M$, we denote by $|M|$ the domain of $M$ (although we will be sloppy in distinguishing between the two), and by $||M||$ the cardinality of $M$.
	\item Given a $\tau$-structure $M$ and $A \subseteq M$, we denote by $Aut(M/A)$ the set of automorphisms of $M$ which are the identity on $A$.
	\item We denote by $\mathfrak{L}_{\omega_1, \omega}(\tau)$ the logical language $\mathfrak{L}_{\omega_1, \omega}$ (admitting countable disjunctions and countable conjuctions) with respect to the vocabulary $\tau$.
	\item Given a collection $\Delta$ of formulas in one free variable of the language $\mathfrak{L}_{\omega_1, \omega}(\tau)$, a $\tau$-structure $M$, and $a \in M$, we let:
	$$tp_{\Delta}(a, \emptyset, M) = \{ \varphi \in \Delta : M \models \varphi(a)\}.$$
\end{enumerate}
\end{notation}

	\begin{definition}\label{semi_interpretability} Let:
	\begin{enumerate}[(i)]
	\item $\tau_\ell$ ($\ell = 1, 2$) be relational languages;
	\item $\Delta_{\ell} \subseteq \mathfrak{L}_{\omega_1, \omega}(\tau_{\ell})$ be sets of formulas, for $\ell = 1, 2$;
	\item $\Delta_2 = \{ \varphi \in \mathfrak{L}_{\omega_1, \omega}(\tau_{2}): \varphi \text{ is an atomic $\tau_2$-formula in one free variable} \}$;
	\item $M_\ell$ be $\tau_{\ell}$-structures, for $\ell = 1, 2$.
	\end{enumerate}
We say that $M_2$ is $\Delta_1$-interpretable in $M_1$ by the scheme $\mathfrak{s}$ and function $\bar{F}$ when:
	\begin{enumerate}[(A)]
	\item $\mathfrak{s} = \{ \mathfrak{s}(p) : p \in \mathfrak{S}_{M_2}\} \cup \{ \mathfrak{s}(R, \bar{p}) : R \in \tau_2, \bar{p} = (p_{\ell} : \ell < k(R)) \in \mathfrak{S}_{M_2}^{k(R)}\}$, where:
	\begin{enumerate}[(a)]
	\item $\mathfrak{S}_{M_2} = \{ tp_{\Delta_2}(a, \emptyset, M_2) : a \in M_2\}$;
	\item $\mathfrak{s}(p) = (r_p(\bar{x}_{m(p)}), E_p(\bar{y}_{m(p)}, \bar{z}_{m(p)})) \in \Delta_1 \times \Delta_1$, $m(p) < \omega$, and $E_{p}(M_1)$ is a non-empty equivalence relation on $r_p(M_1)$;
	\item $\mathfrak{s}(R, \bar{p})$ is a $\tau_1$-formula from $\Delta_1$ of the form $\varphi_{(R, \bar{p})}(\bar{x}^0_{m(p_0)}, ..., \bar{x}^{k-1}_{m(p_{k-1})})$, with $\bar{x}^i_{m(p_i)} = (x^i_0, ..., x^i_{m(p_i)-1})$, for every $i < k = k(R)$;
	\end{enumerate}
	\item $\bar{F} = (F_p : p \in \mathfrak{S}_{M_2})$, where:
	\begin{enumerate}[(a)]
	\item $F_p$ is a one-to-one function from $p(M_2) = \{ a \in M_2 : p = tp_{\Delta_2}(a, \emptyset, M_2) \}$ onto $r_p(M_1)/E_p(M_1)$;
	\item for every predicate $R$ of $\tau_2$ we have: if $k = k(R)$, $\bar{a} \in M_2^{k}$, and, for every $\ell < k$, $p_{\ell} = tp_{\Delta_2}(a_{\ell}, \emptyset, M_2)$, $\bar{b}_{\ell} \in r_{p_{\ell}}(M_{1})$ and $F_{p_{\ell}}(a_{\ell}) = \bar{b}_{\ell}/E_{p_{\ell}}(M_1)$, then:
	$$M_2 \models R(a_0, ..., a_{k-1}) \text{ iff } M_1 \models \varphi_{(R, \bar{p})}(\bar{b}_0, ..., \bar{b}_{k-1}).$$
	\end{enumerate}
\end{enumerate}
Finally, we say that $M_2$ is $\mathfrak{L}_{\omega_1, \omega}$-semi-interpretable in $M_1$ when $M_2$ is $\Delta_1$-interpretable in $M_1$ by the scheme $\mathfrak{s}$ and function $\bar{F}$ for some $\Delta_1$, $\mathfrak{s}$ and $\bar{F}$.
\end{definition}

	The proof of the following fact is essentially as in the case of first-order interpretability (cf. Definition \ref{classical_interpretability}).

	\begin{fact}\label{semi_int_fact} Let $M$ and $N$ be models, and suppose that $N$ is $\mathfrak{L}_{\omega_1, \omega}$-semi-interpretable in $M$. Then every $\pi \in Aut(M)$ induces a $\hat{\pi} \in Aut(N)$, and the mapping $\pi \mapsto \hat{\pi}$ is a continuous homomorphism of $Aut(M)$ into $Aut(N)$. 
\end{fact}

	The following fact is well-known.

	\begin{fact}\label{cont_homo} Let $G$ and $H$ be Polish group and $\alpha: G \rightarrow H$ a group isomorphism. If $\alpha$ is continuous, then $\alpha$ is a topological isomorphism.
\end{fact}

	We are now ready to prove Theorem \ref{main_th}.

	\begin{proof}[Proof of Theorem \ref{main_th}] Let $M$ be a countable model. We construct a countable model $N_M = N$ such that:
	\begin{enumerate}[(1)]
	\item\label{auto1} $N$ is $\mathfrak{L}_{\omega_1, \omega}$-semi-interpretable in $M$ (cf. Definition \ref{semi_interpretability});
	\item\label{auto_ind2} for every $\pi \in Aut(N)$ there is a unique $\pi_0 \in Aut(M)$ such that $\pi = \hat{\pi}_0$ (cf. Fact \ref{semi_int_fact});
	\item\label{stable_enumi} $N$ is $\aleph_0$-stable.
\end{enumerate}
Using Facts \ref{semi_int_fact} and \ref{cont_homo}, and items (\ref{auto1})-(\ref{auto_ind2}) above it follows that $Aut(M)$ and $Aut(N)$ are topologically isomorphic, and thus by (\ref{stable_enumi}) we are done.
\newline We then proceed to the construction of a model $N_M = N$ as above. First all notice that without loss of generality\footnote{Recall that in this paper we only consider structures in a countable language.} we can assume that $M$ is a relational structure in a language $\tau(M) = \{ P_{(n, \ell)} : n < n_* \leq \omega, \ell < \ell_n \leq \omega \}$, where the predicates $P_{(n, \ell)}$ are $n$-ary predicates, and, for transparency, we assume that if $M \models P_{(n, \ell)}(\bar{a})$, then $\bar{a}$ is without repetitions. We construct a structure $N$ in the following language $\tau(N)$:
\begin{enumerate}[(i)]
	\item $c \in \tau(N)$ is a constant;
	\item $P \in \tau(N)$ is a unary predicate;
	\item for $n < n_* \leq \omega$ and $\ell < \ell_n \leq \omega$, $Q_{(n, \ell)} \in \tau(N)$ is a unary predicate;
	\item for $n < n_* \leq \omega$ and $\ell < \ell_n \leq \omega$, $E_{(n, \ell)} \in \tau(N)$ is a binary predicate;
	\item for $n < n_* \leq \omega$, $\ell < \ell_n \leq \omega$ and $\iota < n$, $F_{(n, \ell, \iota)} \in \tau(N)$ is a unary function;
	\item for $n < n_* \leq \omega$, $\ell < \ell_n \leq \omega$ and $j < \omega$, $G_{(n, \ell, j)} \in \tau(N)$ is a unary function.
\end{enumerate}
We define the structure $N$ as follows:
\begin{enumerate}[(a)]
	\item $|N|$ (the domain of $N$) is the disjoint union of:
	$$P^N \cup \{ c^N = e \} \cup \{ Q^N_{(n, \ell)} : n < n_* \leq \omega \text{ and } \ell < \ell_n \leq \omega \};$$
	\item $P^N = |M|$ (the domain of $M$);
	\item $Q^N_{(n, \ell)} = \{ (n, \ell, i, a_0, ..., a_{n-1}) : a_{t} \in M, i \leq \omega, (a_0, ..., a_{n-1}) \notin P^{M}_{(n, \ell)} \Rightarrow i < \omega \}$;
	\item $E^N_{(n, \ell)} = $
	$$\{ ((n, \ell, i_1, \bar{a}), (n, \ell, i_2, \bar{a})) : i_1, i_2 \leq \omega, (n, \ell, i_{t}, \bar{a} = a_0, ..., a_{n-1}) \in Q^N_{(n, \ell)} \};$$
	\item for $\iota < n$, $F_{(n, \ell, \iota)}(x) = \begin{cases} 
a_{\iota} \;\;\; \text{ if } x = (n, \ell, i, a_0, ..., a_{n-1}),\\
e \;\;\;\;\, \text{ otherwise};
\end{cases}$
\item for $j < \omega$, $G_{(n, \ell, j)}(x) = \begin{cases} 
(n, \ell, j, a_0, ..., a_{n-1}) \;\;\; \text{ if } x = (n, \ell, i, a_0, ..., a_{n-1}),\\
e \;\;\;\;\;\;\;\;\;\;\;\;\;\;\;\;\;\;\;\;\;\;\;\;\;\;\;\;\;\;\;\, \text{ otherwise}.
\end{cases}$
\end{enumerate}
We now prove items (1)-(3) from the list at the beginning of the proof. Item (3) is proved in Claim \ref{claim_stability}.
\noindent We prove item (2). Let $\pi \in Aut(N)$ and, for $a, b \in M$, let $\pi_0(a) = b$ iff $\pi(a) = b$. Clearly $\pi_0 \in Sym(M)$. For the sake of contradiction, suppose that $\pi_0 \notin Aut(M)$. Replacing $\pi$ with $\pi^{-1}$, we can assume without loss of generality that there are $n < n_* \leq \omega$, $\ell < \ell_n \leq \omega$, $\bar{a} = (a_0, ..., a_{n-1}) \in M^{n}$ and $\bar{b} = (b_0, ..., b_{n-1}) \in M^{n}$ such that $\pi_0(\bar{a}) = \bar{b}$, $M \models P_{(n, \ell)}(\bar{a})$ and $M \not\models P_{(n, \ell)}(\bar{b})$. Then the element $(n, \ell, \omega, a_0, ..., a_{n-1}) \in N$ realizes the type:
$$p = \{ F_{(n, \ell, \iota)}(x) = a_\iota : \iota < n\} \cup \{ G_{(n, \ell, j)}(x)\neq x : j < \omega \},$$
while the type:
$$q = \{ F_{(n, \ell, \iota)}(x) = b_\iota : \iota < n\} \cup \{ G_{(n, \ell, j)}(x)\neq x : j < \omega \},$$
is not realized in $N$, a contradiction. Hence, $\pi_0 \in Aut(M)$ and, easily, $\pi = \hat{\pi}_0$ (cf. Fact \ref{semi_int_fact}) and for every  $\pi_1 \in Aut(M)$ such that $\pi = \hat{\pi}_1$ we have that $\pi_0 = \pi_1$.
\newline Finally, we prove item (1). Let $(k_{(n, \ell, i)} : n < n_* \leq \omega, \ell < \ell_n \leq \omega,  i \leq \omega)$ be a sequence of natural numbers such that:
	$$(n_1, \ell_1, i_1) \neq (n_2, \ell_2, i_2) \text{ implies }1 < n_1 + k_{(n_1, \ell_1, i_1)} \neq n_2 + k_{(n_2, \ell_2, i_2)}.$$
Let also:
\begin{enumerate}[(i')]
	\item  $n + k_{(n, \ell, i)} = m(n, \ell, i)$;
	\item $\bar{x}_{m(n, \ell, i)} = (x_0, ..., x_{m(n, \ell, i) -1})$;
	\item  $\bar{y}_{m(n, \ell, i)} = (y_0, ..., y_{m(n, \ell, i) -1})$.
\end{enumerate}
Consider now the following formulas:
	\begin{enumerate}[(A)]
	\item $\varphi_0(x_0) : x_0 = x_0$;
	\item $\theta_0(x_0, y_0) : x_0 = y_0$;
	\item for $n < n_* \leq \omega$, $\ell < \ell_n \leq \omega$ and $i < \omega$ let:
	$$\varphi_{(n, \ell, i)}(\bar{x}_{m(n, \ell, i)}) : \bigwedge_{m < m(n, \ell, i)} x_m = x_m,$$
	$$\theta_{(n, \ell, i)}(\bar{x}_{m(n, \ell, i)}, \bar{y}_{m(n, \ell, i)}) : \bigwedge_{m < n} x_m = y_m;$$
	\item for $n < n_* \leq \omega$, $\ell < \ell_n \leq \omega$ and $i = \omega$ let:
	$$\varphi_{(n, \ell, i)}(\bar{x}_{m(n, \ell, i)}) : \bigwedge_{m < m(n, \ell, i)} x_m = x_m \wedge P_{(n, \ell)}(x_0, ..., x_{n-1}),$$
	$$\theta_{(n, \ell, i)}(\bar{x}_{m(n, \ell, i)}, \bar{y}_{m(n, \ell, i)}) : \bigwedge_{m < n} x_m = y_m.$$
\end{enumerate}
Notice now, that:
	\begin{enumerate}[(I)]
	\item $P^N = \varphi_0(M)/\theta_0(M)$;
	\item $Q^N_{n, \ell}$ is in bijection with $\bigcup \{ \varphi_{(n, \ell, i)}(M)/\theta_{(n, \ell, i)}(M) : i \leq \omega \}$.
\end{enumerate}
	Using this observation it is easy to see how to choose $\Delta_{M}$, $\mathfrak{s}$, and $\bar{F} = (F_p : p \in \mathfrak{S}_{N})$ as in Definition \ref{semi_interpretability} so as to witness that $N$ is $\mathfrak{L}_{\omega_1, \omega}$-semi-interpretable in $M$.
\end{proof}

	\begin{cclaim}\label{claim_stability} Let $N$ be as in the proof of Theorem \ref{main_th}. Then $Th(N)$ is $\aleph_0$-stable.	
\end{cclaim}

	\begin{proof} Let $N_1$ be a countable model of $Th(N)$. It is enough to show that  there are only countably many $1$-types over $N_1$. To this extent, let $N_2$ be an $\aleph_1$-saturated model of $Th(N)$ such that every countable non-algebraic type is realized by $||N_2||$-many elements, and define the following equivalence relation $E^* = E^*_{(N_1, N_2)}$ on $N_2$:
	$$a E^* b \text{ iff } \exists \pi \in Aut(N_2/N_1) \text{ such that } \pi(a) = b.$$
We will show that the relation $E^*$ has $\aleph_0$ equivalence classes, clearly this suffices. To this extent, notice that:
\begin{enumerate}[$(\star_1)$]
\item if $\pi$ is a permutation of $P^{N_2}$ which is the identity on $P^{N_1}$, then there is an automorphism $\check{\pi}$ of $N_2$ over $N_1$ extending it (recall that $N_2$ is $\aleph_1$-saturated);
\end{enumerate}
\begin{enumerate}[$(\star_2)_{(n, \ell)}$]
\item if $b_1, b_2 \in E_{(n, \ell)}$, $(F_{n, \ell, \iota}(b_1) : \iota < n)$ and $(F_{n, \ell, \iota}(b_2) : \iota < n)$ realize the same $\{=\}$-type over $P^{N_1}$, and for $t = 1, 2$ we have $b_t \notin \{ G_{(n, \ell, j)}(b_{t}) : j < \omega\}$, then there exists $\pi \in Aut(N_2/N_1)$ such that $\pi(b_1) = b_2$;
\end{enumerate}
\begin{enumerate}[$(\star_3)_{(n, \ell, j)}$]
\item if $b_1, b_2 \in E_{(n, \ell)}$, $(F_{n, \ell, \iota}(b_1) : \iota < n)$ and $(F_{n, \ell, \iota}(b_2) : \iota < n)$ realize the same $\{=\}$-type over $P^{N_1}$, and for $t = 1, 2$ we have $G_{(n, \ell, j)}(b_{t}) = b_{t}$, then there exists $\pi \in Aut(N_2/N_1)$ such that $\pi(b_1) = b_2$.
\end{enumerate}
Now, using $(\star_1)$-$(\star_2)_{(n, \ell)}$-$(\star_3)_{(n, \ell, j)}$ and noticing that $n, \ell$ and  $j$ range over countable sets, it is easy to see that the relation $E^*$ defined above has $\aleph_0$ equivalence classes.
\end{proof}

\end{document}